\documentclass{amsart}
\usepackage{amscd, amssymb, mathrsfs, mathabx, 
dsfont, xcolor, amsmath, soul,diagrams}

\usepackage{hyperref}  

\usepackage{collectbox}
\usepackage{enumitem}

\usepackage{tikz}
\usetikzlibrary{shapes, arrows}



\input{xy}
\xyoption{all}

\newtheorem*{theorem*}{Main Theorem}
\newtheorem*{lemma*}{Lemma A}

\newtheorem{theorem}{Theorem}[section]
\newtheorem{proposition}{Proposition}
\newtheorem{corollary}[theorem]{Corollary}
\newtheorem{lemma}[theorem]{Lemma}

\theoremstyle{definition}
\newtheorem{definition}{Definition}

\theoremstyle{remark}
\newtheorem{remark}[theorem]{Remark}

\numberwithin{equation}{section}
\newcommand{\ti }{\times}

\newcommand{\PP }{{\mathbb P}}

\newcommand{\ZZ }{{\mathbb Z}}
\newcommand{\RR }{{\mathbb R}}

\newcommand{\ka }{{\alpha}}

\newcommand{\lan}{\langle}
\newcommand{\ran}{\rangle}

\newcommand{\Z}{\mathbb{Z}}

\newcommand{\Ch}{\mathrm{Ch}}
\newcommand{\str}{\mathrm{str}}
\newcommand{\sdet}{\mathrm{sdet}}

\newcommand{\HN}{HN}

\newcommand{\OHN}{\overline{HN}}

\newcommand{\chhh}{\mathrm{ch}_{\mathrm{HH}}}

\newcommand{\chhn}{\mathrm{ch}_{\mathrm{HN}}}
\newcommand{\chhp}{\mathrm{ch}_{\mathrm{HP}}}
\newcommand{\chhnx}{\mathrm{ch}_{\mathrm{HN}}^{\bar{X}}}

\newcommand{\chhhx}{\mathrm{ch}_{\mathrm{HH}}^{\bar{X}}}

\newcommand{\fm}{\mathfrak{m}}
\newcommand{\MF}{\mathrm{MF}}
\newcommand{\id}{\operatorname{id}}
\newcommand{\End}{\mathrm{End}}

\newcommand{\Spec}{{\mathrm{Spec}}}

\newcommand{\cQ}{{\mathcal{Q}}}
\newcommand{\cA}{{\mathcal{A}}}

\newcommand{\cE}{{\mathcal{E}}}

\def\on{\operatorname}

\def\Gres{\on{Res}^G}

\def\res{\on{res}}

\newcommand{\tr }{{\mathrm{tr}}}

 \newcommand{\fU}{\mathfrak{U}}

\newcommand{\xra}[1]{\xrightarrow{#1}}
\newcommand{\xla}[1]{\xleftarrow{#1}}

\newcommand{\ot }{\otimes}
\newcommand{\cO}{{\mathcal{O}}}

\newcommand{\ccan}{{\mathsf{can}}}

\newcommand{\C}{{\mathbb{C}}}

\newcommand{\rH}{{\mathrm{H}}}
\newcommand{\cH}{{\mathsf{H}}}

\newcommand{\uu}{{(\!(u)\!)}}
\newcommand{\uuu}{{[\![u]\!]}}

\newcommand{\R}{{\mathbb{R}}}

\def\bu{\bullet}

\newcommand{\cC}{{\mathcal{C}}}

\def\on{\operatorname}

\newcommand{\bHom}{{\mathbf{Hom}}}

\newcommand{\Hom }{{\mathrm{Hom}}}

\newcommand{\op}{{\mathrm{op}}}

\newcommand{\Perf}{\mathrm{Perf}}

\newcommand{\HH}{{\mathbb H}}

\newcommand{\td}{\mathrm{td}_{\mathrm{HP}}}

\newcommand{\tdhh}{\mathrm{td}_{\mathrm{HH}}}
\newcommand{\tdx}{\mathrm{td}_{\mathrm{HP}}^{\bar{X}}}
\newcommand{\tdhhx}{\mathrm{td}_{\mathrm{HH}}^{\bar{X}}}

\newcommand{\Hoch}{\mathsf{Hoch}}

\newcommand{\MC}{\mathrm{MC}}

\def\oa{\overline{a}}

\def\trace{\mathsf{trace}}

\newcommand{\sh}{{\mathsf{sh}}}
\newcommand{\csh}{{\mathsf{csh}}}
\newcommand{\SH}{{\mathsf{Sh}}}

\newcommand{\lrarrow}[2] {\stackrel{{\rTo^{#1}}}{{\lTo_{#2}}}}

\makeatletter
\@namedef{subjclassname@2020}{%
  \textup{2020} Mathematics Subject Classification}
\makeatother

\begin{document}

\title[Canonical pairing and Hirzebruch-Riemann-Roch formula]
      {Canonical pairing and Hirzebruch-Riemann-Roch formula for matrix factorizations}

\author[H. Kim]{Hoil Kim}
\address{Department of Mathematics\\
  Kyungpook National University\\
  Taegu 702-701, Republic of Korea}
\email{hikim@knu.ac.kr}

\author[T. Kim]{Taejung Kim }
\address{Department of Mathematics Education\\
  Korea National University of Education\\
250 Taeseongtabyeon-ro, Gangnae-myeon\\
Heungdeok-gu, Cheongju-si, Chungbuk 28173\\
 Republic of Korea}
\email{tjkim@kias.re.kr}


\thanks{H. Kim was supported by NRF-2018R1D1A1B07044575.}

\begin{abstract}
We formulate a realization of the canonical pairing  in the negative cyclic homology of the category of local matrix factorizations and for global matrix factorizations, by introducing a twisted de Rham valued Todd class we establish a formula of the Hirzebruch-Riemann-Roch theorem in the case of periodic cyclic homology.
\end{abstract}


\maketitle

 \section{Introduction}

Let $\cA$ be a $\C$-linear proper differential graded (for short dg) category, i.e.,
$$\sum_i\dim \rH^i(\Hom_\cA(-,-))<\infty.$$
A categorical canonical  pairing on $HP(\cA)\times HP(\cA^\op)$ is by definition a composition of the following maps:
\begin{equation}\label{defcann}
  \langle-,-\rangle_{HP(\cA)\ot HP(\cA^\op)}:=HP(\bHom_\cA)\circ \mathsf{kun}
  \end{equation}
where $\mathsf{kun}$ is the K\"{u}nneth isomorphism
$$HP(\cA\otimes 1)\otimes_{HP(\Perf (\C))}HP(1\otimes \cA^\op)\simeq HP(\cA\otimes\cA^\op)$$
and $\bHom_\cA: \cA\otimes\cA^\op\to\Perf (\C)$ is given by $N\otimes M\mapsto\Hom_{\cA}(M,N)$.

Let $Q:=\mathbb{C}[x_1,\dots,x_n]$ and assume that $\fm := (x_1, \dots, x_n)$ 
is the only critical point of the map $f: \mathbb{A}^n_\C \to \mathbb{A}^1_\C$. Then the $\Z/2$ graded dg category $\MF(Q,f)$ of matrix factorizations is proper and there is a natural embedding
$$HN(\MF(Q,f))\hookrightarrow HP(\MF(Q,f)).$$
After composing an isomorphism between $HN(\MF(Q,f))$ and $\HN(\MF(Q,-f))$, one can consider the restriction of the canonical pairing on the negative cyclic homology, i.e., on $HN(\MF(Q,f))\times HN(\MF(Q,f))$. Denoting it  by $\langle-,-\rangle_{\ccan}$, one may ask how much explicitly one can describe this canonical pairing. Let
\begin{equation}
  \begin{aligned}
  \cH_{f}^{(0)}&:=\rH_n(\Omega^{\bullet}_{Q/\C}\uuu, -d{f} +ud)\\
  \cH_{-f}^{(0)}&:=\rH_n(\Omega^{\bullet}_{Q/\C}\uuu, d{f} +ud)
  \end{aligned}
\end{equation}
where $u$ is a formal variable of degree 2. In this paper,

\begin{equation}\label{outline}
\xymatrix{\ar @{} [drr]
  \OHN_n(\MF(Q,f)) \times \OHN_n(\MF(Q,f))  \ar[d]^-{\id \times \Psi} \ar[rr]_-\cong && \cH_{f}^{(0)}\times \cH_{f}^{(0)}
\ar[d]^-{\id \times (-1)^n\cdot} \\
\ar @{} [drr] 
\OHN_n(\MF(Q,f)) \times \OHN_n(\MF(Q,-f)) \ar[rr]_-\cong 
\ar[d]^-\star  && \cH_{f}^{(0)}\times \cH_{-f}^{(0)}
\ar[d]^-{\wedge} \\
\ar @{} [drr] 
\OHN_{2n}(\MF^\fm(Q_{\fm},0))  \ar[dr]_-{(-1)^{n(n+1)/2}\trace} \ar[rr]^-{\tr_\nabla^\fm} && \rH_{2n} \R\Gamma_\fm (\Omega_{Q_{\fm}/\C}^\bu\uuu,ud) \ar[dl]^-{\on{\res}} \\
& \C\uuu.&& }
\end{equation}
we prove the commutativity of Diagram~\eqref{outline}, which generalizes the work of M. Brown and M. Walker \cite{BW}. We note that in \cite{BW}, it is proved that the same diagram with $HH_n(\MF(Q,f)$ instead of $\OHN_n(\MF(Q,f))$ is commutative. Moreover, we prove the Hirzebruch-Riemann-Roch formula for a global matrix factorizations by considering the  composition of the global version of right vertical maps, which generalizes the work of B. Kim \cite{Kim}.

One sees that the composition of the left vertical maps in Diagram~\eqref{outline} is nothing but $(-1)^{\frac{n(n+1)}{2}}\cdot\langle-,-\rangle_{\ccan}$. In \cite{HK} we  formulate its explicit formula of the composition of right vertical maps in Diagram~\eqref{outline} and establish that it is a canonical pairing  with multiplication by $(-1)^{\frac{n(n+1)}{2}}$ without relying on the commutativity of Diagram~\eqref{outline} which will be shown in this paper and the explicit formula found in \cite{HK} also enables us to prove that it is Saito's higher residue pairings \cite{ksaito}, which reproves the result of Shklyarov in \cite{shkl}.

The structure of this article is as follows: In Section~\ref{secmc2}, we discuss about relevant definitions and facts to understand the main problem related to Diagram~\eqref{outline}. In Section~\ref{endo25}, we introduce, so-called, a compact generator of the category of matrix factorizations and in Section~\ref{HKR32} we explain three types of Hochschild-Kostant-Rosenberg type chain maps which are the same on the homology level. Based on the arguments in \cite{BW3, shkl, shk3}, we prove the commutativity of the first rectangle and the second rectangle in Diagram~\eqref{outline} in Section~\ref{33first} and Section~\ref{34second} respectively. The commutativity of the third triangle in Diagram~\eqref{outline} is proved in Section~\ref{thirdsec} by generalizing the work in the Hochschild homology of M. Brown and M. Walker \cite{BW3} to the negative cyclic homology. In Section~\ref{normal}, based on the work of B. Kim \cite{Kim} we globalize the local description in Diagram~\eqref{outline}. One needs two important ingredients, the deformation of normal cone and the characteristic equation of canonical pairing in \eqref{eqn: char pairing} which is another viewpoint of the canonical pairing. We note that the local explicit description of  \eqref{eqn: char pairing} is found in \cite{HK}. With this global description of the canonical pairing and introducing a twisted de Rham valued Todd class, we give the proof of the Hirzebruch-Riemann-Roch theorem for the periodic cyclic homology of the category of matrix factorizations based on the categorical version of Hirzebruch-Riemann-Roch theorem in  \cite{Shk: HRR, TK}.

 \section{Matrix factorizations and mixed complexes}\label{secmc2}

 \subsection{The category of matrix factorizations}
Let $Q$ be $\mathbb{C}[x_1,\dots,x_n]$ and $f\in Q$. We define the $\mathbb{Z}/2$-graded dg category $\MF(Q,f)$ of matrix factorizations as follows:
\begin{definition}
  A  matrix factorization of a potential $f$ over $Q$ is a 
pair
\begin{equation}
(E,\delta_E)=\xymatrix{(E^0 \ar@/^/[r]^{\delta_0} & E^1\ar@/^/[l]^{\delta_1})}\text{ where }
\end{equation}
\begin{itemize}
\item  $E=E^0\oplus E^1$ is a $\mathbb{Z}/2$-graded finitely generated projective $Q$-module and
\item $\delta_E\in \End^1_Q(E)$ is an odd (i.e., of degree $1 \in \mathbb{Z}/2$) endomorphism of $E$ such that $\delta_E^2=f\cdot \id_E$.  
\end{itemize}
 \end{definition}
One can also see that a matrix factorization $(E,\delta_E)$ with a potential $f$ is a perfect right cdg $\ZZ/2$-graded module over a curved differential graded algebra $(Q, 0, -f)$. For completeness and reader's convenience, let us recall that one calls $(Q, d_Q, f)$ a  {\em curved differential graded {\rm (\rm for short, cdg)} $k$-algebra} if $Q$ is a unital  $\ZZ$-graded (or $\ZZ/2$-graded) algebra over a field $k$ of characteristic zero with a degree one $k$-linear endomorphism of
$Q$, namely a differential $d_Q$, and a curvature $f\in Q$ with $|f|=2$ such that 
\begin{itemize}
\item[(i)] $d_Q(a_1a_2) = d_Q(a_1) a_2 + (-1)^{|a_1|} a_1 d_Q(a_2)$ for $a_1, a_2 \in Q$;
\item[(ii)] $d_Q ^2 (a) = f a - a f$ for $a \in Q$;
\item[(iii)] $d_Q(f) = 0$. 
\end{itemize}
In particular, if $d_Q=0$, $(Q,0,f)$ is called a {\em curved algebra} and if $f=0$, we call  $(Q, d_Q,0)$, for short $(Q,d_Q)$,
a  {\em dg $k$-algebra}. A right cdg module $(E,\delta_E)$ over $(Q, d_Q, f)$ is a right $\ZZ$-graded (or $\ZZ/2$-graded) $Q$-module $E$
with a $k$-linear map $\delta_E : E \to E$ with $|\delta_E|=1$ such that $\delta_E (m a) = \delta_E  (m) a + (-1)^{|m|} m d_Q (a)$ for every $m\in E$, $a \in Q$, and $\delta^2_E=\rho_{-f}$ where $\rho _{-f}$ is the right multiplication by $-f$. Moreover, if we do not require $\delta^2_E=\rho_{-f}$, then we call it a right cdg {\em quasi-}module $(E,\delta_E)$ over $(Q, d_Q, f)$. One can see that when $Q$ is a commutative $k$-algebra concentrated in degree $0$ and $E$ is  a  finitely generated projective $\Z/2$-graded $Q$-module, a right cdg module $(E, \delta_E)$ over $(Q, 0, -f)$ becomes a {\em matrix factorization} of $f$ over $Q$; see \cite{BW,CKK,TK} for more details.

If  $M$ is a $\ZZ$-graded (or $\ZZ/2$-graded) vector space, $b^2 = 0 = B^2$, and $bB + Bb = 0$ with $|b|=1$ and $|B|=-1$, then  $(M, b, B)$ is a mixed complex. A morphism $\phi: (M,b,B) \to (M',b',B')$ between mixed complexes is a $k$-linear map preserving degrees and both differentials. In particular, $\phi$ is called a {\em quasi-isomorphism} if it is a morphism and a quasi-isomorphism between $(M, b) \to (M' , b')$.

Let  $\cA$ be a cdg $k$-category with curvature $f$ and $\cA (x, y) [1]$ denote the degree shifted by $1$. Letting $\mathbb{X}=(x_0, ..., x_n)$, the Hochschild  complex $\Hoch (\cA)$ of $\cA$ is 
\begin{multline*}
 \Hoch (\cA) :=  \bigoplus _{x \in \cA} \cA (x, x) \oplus \bigoplus _{n\ge 1} \big( \bigoplus _{\mathbb{X}\in \cA ^{\ot n+1} } 
 \cA (x_1, x_0) \ot _k \underbrace{\cA (x_2, x_1) [1] \ot_k ... \ot_k  \cA (x_0, x_n) [1]}_{n}\big)  
\end{multline*} 
with differential $b:= b_2 + b_1 + b_0$ defined as follows:
\begin{equation}\label{eqconv1}
  \begin{aligned} 
b_2 (a_0[a_1 | ... | a_n]) & :=   (-1)^{|a_0|}  a_0 a_1 [ a_2 | ... | a_n] + \sum_{j=1}^{n-1}(-1)^{\sum_{i=0}^{j} |a_i | - j}  a_0[ a_1 | ...|a_ja_{j+1} |...| a_n]  \\
                           &-  (-1) ^{(|a_n|+1 ) (\sum_{i=0}^{n-1} |a_i| - (n-1))} a_n a_0 [ a_1 | ... | a_{n-1} ] ;\\
b_1   (a_0[a_1 | ... | a_n])  & :=   d (a_0) [a_1 | ... | a_n] + \sum_{j=1}^{n} (-1)^{\sum_{i=0}^{j-1}  |a_i| - j} a_0 [ a_1| ... | d (a_j)|...|a_n ] ; \\                    
  b_0    (a_0[a_1 | ... | a_n]) & :=  (-1)^{|a_0|} a_0 [ f | a_1| ... | a_n]  + \cdots  
                         +   (-1)^{ \sum _{i=0}^n |a_i| - n} a_0 [ a_1 | ... | a_{n} | f ] .
\end{aligned}
\end{equation}
We also define a {\em Hochschild complex of the
   second kind} with the differentials $b_i$, $i=0, 1, 2$ as
\begin{multline*}  \Hoch^{II} (\cA) =  \bigoplus _{x \in \cA} \cA (x, x) \oplus \prod _{n\ge 1} \Big( \bigoplus _{\mathbb{X}\in \cA ^{\ot n+1} } 
 \cA (x_1, x_0) \ot _k \underbrace{\cA (x_2, x_1) [1] \ot_k ... \ot_k  \cA (x_0, x_n) [1]}_{n} \Big).
\end{multline*}

Let $D$ be a subcomplex of $\Hoch (\cA)$ generated by elements $a_0[a_1| ... | a_n]$ for which 
$a_i = c \cdot \mathrm{id}_x$ for  $x \in \cA$, $c\in k$ and some $i\ge 1$. We define the normalized Hochschild complex to be
\[ \overline{\Hoch} (\cA) :=  (\Hoch (\cA), b) / D =: \bigoplus _{n\ge 0}   \overline{\Hoch}_n (\cA). \]
The Connes operator descends to an operator on  $\overline{\Hoch} (\cA)$ and is given by  
\begin{equation}\label{eqconv2}
B(a_0[\oa_1| \cdots |\oa_n]) = \sum_{l=0}^n (-1)^{(|a_l| + \cdots +|a_n| - (n-l+1)) (|a_0| + \cdots +|a_{l-1}| - l)   } 
1[\oa_l | \cdots |\oa_n| \oa_0 | \cdots | \oa_{l-1}].
\end{equation}
We let $\overline{\MC} (\cA) := (\overline{\Hoch} (\cA), b, B )$ and $\overline{\MC} ^{II} (\cA) := (\overline{\Hoch}^{II} (\cA), b, B )$. For a dg category $\cA$ the natural map $\MC  (\cA) \xrightarrow{\sim} \overline{\MC} (\cA)$ is a quasi-isomorphism; see \cite[\S 5.3]{Loday} for details and   the natural map $\MC ^{II} (\cA) \xrightarrow{\sim} \overline{\MC}^{II} (\cA)$     is a quasi-isomorphism for a cdg category $\cA$; see \cite[Proposition~3.15]{BW} for details.

We let
 $$\begin{aligned}
   \HN _* (\cA ) &:=  \rH^{-*}  (\Hoch (\cA )\uuu, b + uB)\\
   \OHN _* (\cA ) &:=  \rH^{-*}  (\overline{\Hoch} (\cA )\uuu, b + uB)\\
    HP _* (\cA ) &:=  \rH^{-*}  (\Hoch (\cA )\uu, b + uB)\\
   \end{aligned}$$
where $u$ is a formal variable with degree $2$.

\subsection{Matrix factorization with support}

Let $Q=\mathbb{C}[x_1,\dots,x_n]$, $\fm=(x_1,\dots,x_n)$, and the only critical point of $f:=\sum_{i=1}^{n}x_iw_i\in Q$ be $\fm$. We let  $\MF^Z(Q,f)$ be the full dg-subcategory of $\MF(Q,f)$ consisting of those objects whose supports are in a closed subset $Z\subseteq \Spec(Q)$. In particular when $Z=\Spec (\fm)$, we denote it by  $\MF^\fm(Q,f)$. One has
\begin{equation}\label{support}
  \MF^\fm(Q,f)=\MF(Q,f);\text{ see \cite[Proposition 3.2]{BW3}}.
  \end{equation}

The $k$th local cohomology associated with the maximal ideal $\fm$ of a complex
  $(\mathcal{F}^\bullet,d_{\mathcal{F}^\bullet})$ is by definition the $k$th hypercohomology $\mathbb{H}_k \R\Gamma_{\fm} (\mathcal{F}^\bullet,d_{\mathcal{F}^\bullet})$. Let
  $$\cC(x_1,\dots,x_n)=\bigoplus_{j=0}^{n}\cC^j\text{ with }\cC^j=\oplus_{i_1<\cdots<i_j}Q\big[\frac{1}{x_{i_1}\cdots x_{i_j}}\big]\alpha_{i_1}\cdots \alpha_{i_j}$$
  where $\alpha_{i}^{2}=0$, $\alpha_i\alpha_j=-\alpha_j\alpha_i$, and $|\alpha_i|=1$. Then $(\cC(x_1,\dots,x_n), \sum_{i=1}^{n}\alpha_i)$ is a complex and 
  \begin{equation}\label{loco25}
    (\cC(x_1,\dots,x_n)\otimes_Q\mathcal{F}^\bullet, d_{\mathcal{F}^\bullet}+\sum_{i=1}^{n}\alpha_i)
    \end{equation}
  is an injective resolution of a complex $(\mathcal{F}^\bullet,d_{\mathcal{F}^\bullet})$, which is called the \v{C}ech resolution. The $k$th hypercohomology $\mathbb{H}_k \R\Gamma_{\fm} (\mathcal{F}^\bullet,d_{\mathcal{F}^\bullet})$ can be calculated by the $k$th homology of the totalization of \eqref{loco25}. For example, taking $(\mathcal{F}^\bullet,d_{\mathcal{F}^\bullet})=(\Omega_{Q}^\bu\uuu, -df+ud)$ where $u$ is a formal variable of degree 2, the $k$th hypercohomology $\mathbb{H}_k \R\Gamma_{\fm} (\Omega^\bu_{Q}\uuu, -df+ud)$ can be calculated by the $k$th homology of the totalization of the double complex:
$$(\cC(x_1,\dots,x_n)\otimes_Q\Omega_{Q}^\bu\uuu, -df+ud+\sum_{i=1}^{n}\alpha_i).$$

\begin{definition} \label{def825}
For the $\Z/2$-graded $Q_\fm$-module $\Omega^\bu_{Q_\fm/\C}$, there is a well-defined $\C$-linear map $\overline{\Gres}: \rH^{n}_{\fm}(\Omega^n_{Q_\fm/\C}/d\Omega^{n-1}_{Q_\fm/\C}) \to \C$ induced by Grothendieck's residue map $\Gres: \rH^{n}_{\fm}(\Omega^n_{Q_\fm/\C}) \to \C$ given by
\begin{equation} \label{E527}
\Gres \left[  \frac{dx_1 \cdots dx_n} {x_1^{a_1}, \cdots, x_n^{a_n}} \right] =
\begin{cases} 
1 & \text{if $a_i = 1$ for all $i$, and } \\
0 & \text{otherwise;}
\end{cases}  
\end{equation}
see \cite[Section 5]{kunz}. By abuse of notation, letting its $\C\uuu$-linear extension be $\Gres$ again, we let the following composition
$$
\res:\mathbb{H}_{2n} \R\Gamma_\fm(\Omega^\bu_{Q_\fm/\C}\uuu,ud)  
\to \rH_\fm^n(\Omega^n_{Q_\fm/\C}/d\Omega^{n-1}_{Q_\fm/\C})\uuu \xra{\Gres} \C\uuu.
$$
\end{definition}

\section{The commutativity of Diagram~\eqref{outline}}

\subsection{Koszul matrix factorization}\label{endo25}

Let $Q=\mathbb{C}[x_1,\dots,x_n]$, $\fm=(x_1,\dots,x_n)$, and the only critical point of $f:=\sum_{i=1}^{n}x_iw_i\in Q$ be $\fm$.
One can define a $\Z/2$-graded Koszul matrix factorization $K_f := \mathsf{Kos}_{Q}(x_1, \dots, x_n)$ to be the exterior algebra over $Q$ generated by odd degree elements $e_1,\dots,e_n$ with differential
$$d_{K_f}:=\sum_{i=1}^{n}x_ie^\ast_i+w_ie_i$$
where $e^\ast_i$ is the dual of $e_i$. Note that $(K_f,d_{K_f})\in\MF(Q,f)$. We let $\cE_f:= \End_{\MF(Q, f)}(K_f)$ be its $\Z/2$-graded endomorphism algebra with differential $d^{\cE_f}:=[d_{K_f},-]$. Note that $(d^{\cE_f})^2=0$, i.e., it is a dg algebra. According to \cite[Theorem 5.2 (3)]{dyckerhoff} and \cite[Section 3.3]{BW3}, one has an isomorphism
\begin{equation}\label{s213}
  \OHN(\cE_f,d^{\cE_f})\cong\OHN(\MF(Q,f)).
\end{equation}

Let us consider the case for $f\equiv 0$ and a local ring $Q_\fm$. Set a $\Z/2$-graded $K := \mathsf{Kos}_{Q_\fm}(x_1, \dots, x_n)$ with differential
$$d_{K}=\sum_{i=1}^{n}x_ie^\ast_i$$
and its $\Z/2$-graded endomorphism algebra $\cE:= \End_{\MF^\fm(Q_\fm, 0)}(K)$ with differential $d^{\cE}=[d_K,-]$. Note that
$(K,d_{K})\in\MF^\fm(Q_\fm,0)$. Let $\Lambda$ be the dg $\C$-subalgebra of $\cE$ generated by the $e_i^*$. The inclusion $\Lambda \subseteq \cE$ is a quasi-isomorphism of differential $\Z/2$-graded $\C$-algebras. From \cite[Lemma 4.2]{BW3} whose proof also works for the negative cyclic homology, one can see that there are canonical quasi-isomorphisms
\begin{equation}
\label{213}
\OHN(\Lambda) \xra{\simeq} \OHN(\cE) \xra{\simeq} \OHN(\MF^\fm(Q_\fm,0)) \xleftarrow{\simeq} \OHN(\MF^\fm(Q, 0)).
\end{equation}
induced by an inclusion $(\cE,d^\cE)\in \MF^\fm(Q_\fm,0)$ where $\OHN(\cE)$ is a shorter notation of $\OHN(\cE,d^\cE)$.

\subsection{Hochschild-Kostant-Rosenberg type maps}\label{HKR32}

In this section, we introduce Hochschild-Kostant-Rosenberg type (for short, HKR) maps from $\OHN(\MF(Q,f))$ to $(\Omega^\bullet_{Q/\C}\uuu,ud-df)$ in \cite{BW3,CKK,shkl}. In \cite{shkl,shk3}, composing the $\C\uuu$-linear extended Segal map $\exp(-d^{\cE_f})$ in \cite{Se} with the $\C\uuu$-linear extended classical HKR map $\mathsf{HKR}$ under the identification of $\OHN(\MF(Q,f))$ with  $\OHN(\cE_f)$ in Section~\ref{endo25} by the Morita equivalence, Shklyarov constructs a HKR type map $I_f:\OHN(\MF(Q,f))\to (\Omega^\bullet_{Q/\C}\uuu,ud-df)$ where
\begin{multline}\label{shkmap1}
  \exp(-d^{\cE_f}):(\overline{\Hoch(\cE_f)}\uuu,uB+b_2+b_1)\to  (\overline{\Hoch(\cE_f))}^{II}\uuu,uB+b_2) \\
  \alpha_0 [\alpha_1 | \cdots | \alpha_n] \mapsto
 \sum_{J=0}^\infty \sum_{j_0+\cdots+j_n=J}  (-1)^{J}   \ka _0 [(d^{\cE_f})^ { j_0 } |\ka _1| (d^{\cE_f})^{j_{1}} |\ka _2|\cdots
   |(d^{\cE_f})^{j_{n-1}} | \ka _n|(d^{\cE_f})^{j_n}]  
\end{multline}
and 
\begin{multline}\label{shkmap2}
  \mathsf{HKR}: (\overline{\Hoch(\cE_f)}^{II}\uuu,uB+b_2)\to  (\Omega^\bullet_{Q/\C}\uuu,ud-df)\\
 \alpha_0 [\alpha_1 | \cdots | \alpha_n]  \mapsto
\frac{1}{n!}\str( \alpha_0 \wedge d\alpha_1\wedge \cdots \wedge d \alpha_n).
\end{multline}
We note that \eqref{shkmap1} and \eqref{shkmap2} are chain maps. On the other hand, Brown and Walker define another HKR type map
$\epsilon_{Q,f}:\OHN^{II}(\MF(Q,f))\to (\Omega^\bullet_{Q/\C}\uuu,ud-df)$ in \cite{BW, BW3} as follows. Note that in \cite{BW3}, $\epsilon_{Q,f}$ is defined on Hochschild homology but from \cite{BW}, one can deduce the following. Recall that given objects $(E_0, \delta_0), \dots, (E_n, \delta_n)$ of $q\MF(Q,f)$ where $q\MF(Q,f)$ is the category of quasi-matrix factorizations and maps
$$
E_0 \xla{\alpha_0} E_1 \xla{\alpha_1} \cdots \xla{\alpha_{n-1}} E_n  \xla{\alpha_n} E_0,
$$
there is a (non-strict) cdg-functor $(F, \beta): q\MF(Q,f) \to q\MF(Q,f)^0$ given by $F(E_i, \delta_{i}) = (E_i,0)$ and
$\beta_{(E_i,\delta_i)} = \delta_i$ where $q\MF(Q,f)^0$ is the category of quasi-matrix factorizations with the trivial differentials. The $\C\uuu$-linear extended induced map
$$
(F,\beta)_*: \OHN^{II}(q\MF(Q,f)) \to \OHN^{II}(q\MF(Q,f)^0)
$$
sends $\alpha_0[\alpha_1| \cdots | \alpha_n]$ to 
$$
\sum_{J=0}^{\infty}\sum_{j_0+\cdots+j_n=J}  (-1)^J\alpha_0 [(\delta_1)^{j_0}| \alpha_1 |(\delta_2)^{j_1}| \cdots | \alpha_n | (\delta_0)^{j_n}].
$$
Take a connection $\nabla_{E_i}$ on $E_i$ and set $\nabla_i = \nabla_{E_i}$ (with $\nabla_{n+1} = \nabla_{0}$) for simplicity.
Then a $\C\uuu$-linear map $\epsilon: \OHN^{II}(q\MF(Q,f)^0)\to (\Omega^\bullet_{Q/\C}\uuu,ud-df)$ is defined to be 
$$
\epsilon(\alpha_0[\alpha_1 | \dots | \alpha_n]) 
:=\sum_{J=0}^\infty \sum_{j_0 + \cdots+  j_n = J}  (-1)^{J}   \str\big( \frac{ \ka _0 \nabla_{1}^ { 2j_0 } \ka _1' \nabla_{2}^{2j_{1}} \ka _2'\cdots  \nabla_{n}^{2j_{n-1}}  \ka _n' \nabla_{0}^{2j_n}  }{(n+J )!}\big)u^J,
$$
where $\alpha_i' := \nabla \circ_i \alpha_i - (-1)^{|\alpha_i|} \alpha_i \circ \nabla_{i+1}$ and $\str$ means the supertrace. Then $\epsilon_{Q,f}$ is defined by composing a quasi-isomorphism $\OHN^{II}(\MF(Q,f))\cong\OHN^{II}(q\MF(Q,f))$ with $\epsilon\circ(F,\beta)_\ast$. Note that $\epsilon_{Q,f}$ is independent of the choices of connections at the level of homology and it is compatible with Shklyarov map $I_f$; see \cite[Lemma 3.11]{BW3}.

Under the assumptions and notations in Section~\ref{endo25}, consider a connection 
$$\nabla: K_f \to K_f \otimes_Q \Omega^1_{Q/\C}.$$
We define a  $\C\uuu$-linear Hochschild-Kostant-Rosenberg  type map
$$\tr_\nabla: (\overline{\Hoch(\cE_f)}\uuu,uB+b) \to (\Omega^\bullet_{Q/\C}\uuu,ud-df)$$ 
by sending
\begin{equation} \label{eqn: tr formula}
 \alpha_0 [\alpha_1 | \cdots | \alpha_n]  \mapsto
\sum_{J=0}^\infty \sum_{j_0 + \cdots+  j_n = J}  (-1)^{J}   \str\big( \frac{ \ka _0 R^ { j_0 } \ka _1' R^{j_{1}} \ka _2'\cdots  R^{j_{n-1}}  \ka _n' R^{j_n}  }{(n+J )!} \big)
\end{equation}
where $\alpha_i' := \nabla \circ \alpha_i - (-1)^{|\alpha_i|} \alpha_i \circ \nabla$ and $R  :=  u \nabla ^2  +  (d^{\cE_f})'$.
It is a chain map; see \cite[Theorem 5.2]{CKK} and it is independent of choosing a connection $\nabla$ at the level of homology; see \cite[Theorem 5.4]{CKK} and \cite[The proof of Corollary 5.22]{BW}. In particular, it induces a map 
\begin{equation}\label{local213}
  \R\Gamma_Z\tr_\nabla:\R\Gamma_Z(\Hoch(\cE_f)\uuu,uB+b) \to \R\Gamma_Z(\Omega^\bullet_{Q/\C}\uuu,ud-df)
\end{equation}
which we denote by $\tr_\nabla^Z:=\R\Gamma_Z\tr_\nabla$ where $\R\Gamma_Z$ is the right derived local cohomology functor with support in $Z$. In particular, when $Z$ corresponds to a maximal ideal $\fm$, we denote it by $\tr_\nabla^\fm$.

Replacing the negative cyclic homology with the Hochschild homology in Diagram~\eqref{outline} and using $\epsilon_{Q,f}$,
Brown and Walker prove the commutativity of Diagram~\eqref{outline}. In this paper, we prove the commutativity of Diagram~\eqref{outline} using
$\tr_\nabla$ instead of $\epsilon_{Q,f}$. The proofs for the commutativity of the first rectangular diagram and the second rectangular diagram are essentially in \cite{shkl,shk3}. For completeness and reader’s convenience, we give some sketch of the proofs.

\subsection{The first rectangular diagram}\label{33first}

To a matrix factorization $P=(P_0 \lrarrow{{\delta_0}}{{\delta_1}} P_1)\in\MF(Q,f)$, we associate the dual matrix factorization 
\begin{equation}\label{dual-mf-def}
D(P):=((P_0)^* \lrarrow{{\delta_1^*}}{{-\delta_0^*}} (P_1)^*)\in\MF(Q,-f)
\end{equation}
where $P^*:=\Hom_Q(P,Q)$. It gives a duality functor $D$ which determines an isomorphism of dg categories
$$
D: \MF(Q,f)^\op \xra{\cong} \MF(Q, -f).
$$
Note that it sends an element $\alpha\in\Hom(P_2, P_1)^\op = \Hom(P_1, P_2)$
to the element $\alpha^*\in\Hom(P_2^*, P_1^*)$ where $\alpha^*(\xi) = (-1)^{|\alpha||\xi|} \xi \circ \alpha$.
From \cite[Proposition 3.2]{Shk: NC} and  \cite[Proposition 6]{shkl} under the identification in \eqref{s213}, there is a canonical $\C\uuu$-linear isomorphism of complexes
\begin{equation} \label{E1215c}
\Phi: \OHN(\MF(Q,f)) \xra{\cong} \OHN(\MF(Q,f)^\op)
\end{equation}
given by
$$
a_0[a_1| \cdots |a_n] \mapsto (-1)^{n + \sum_{1 \leq i < j \leq n} (|a_i| - 1)(|a_j| -1)} a^\op_0[a^\op_n| \cdots | a^\op_1].
$$
Composing $\Phi$ with $D$ where $D$ is the duality functor, we obtain a $\C\uuu$-linear isomorphism of complexes
\begin{equation} \label{E1215d}
\Psi: \OHN(\MF(Q,f)) \xra{\cong}  \OHN(\MF(Q,-f))
\end{equation}
given explicitly by
$$
\Psi(a_0[a_1| \cdots |a_n]) = 
(-1)^{n + \sum_{1 \leq i < j \leq n} (|a_i| - 1)(|a_j| -1)} a^*_0[a^*_n| \cdots | a^*_1].
$$
Note that $\Psi$ and $D$ are $\C\uuu$-linear and chain maps in the negative cyclic complexes; see \cite[Proposition 6]{shkl}. The following lemma for Hochschild homology is proved in \cite[Lemma 3.14]{BW3}.

\begin{lemma}\label{lemma21}
Assume that the only critical point of $f\in Q$ is $\fm$. The diagram
$$
\xymatrix{
\OHN(\MF^\fm(Q,f)) \ar[d]^-{\Psi} \ar[r]^-{\tr_{\nabla}^\fm} & \R\Gamma_\fm (\Omega_{Q/\C}^\bu\uuu, ud-df) \ar[d]^-\gamma \\
\OHN(\MF^\fm(Q,- f)) \ar[r]^-{\tr_{\nabla}^\fm} & \R\Gamma_\fm  (\Omega_{Q/\C}^\bu\uuu, ud+df) \\
}
$$
commutes at the level of homology where $\gamma$ is $\R\Gamma_\fm$ applied to the map whose restriction to $\Omega^j_{Q/\C}\uuu$ is multiplication by $(-1)^j$ for all $j$. 
\end{lemma}

\begin{proof}
  Under the identifications in \eqref{support} and \eqref{s213}, it suffices to show that
\begin{equation}\label{xy1}
\xymatrix{
\OHN(\cE_f,d^{\cE_f}) \ar[d]^-{\Psi} \ar[r]^-{\tr_{\nabla}} & (\Omega_{Q/\C}^\bu\uuu, ud-df) \ar[d]^-\gamma \\
\OHN(\cE_f,d^{\cE_f}) \ar[r]^-{\tr_{\nabla}} & (\Omega_{Q/\C}^\bu\uuu, ud+df).
}
\end{equation}
From \cite[Theorem 5.4]{CKK} and \cite[The proof of Corollary 5.22]{BW}, at the level of homology, one sees that
$$\tr_\nabla: (\overline{\Hoch}(\cE_f)\uuu,uB+b) \to (\Omega^\bullet_{Q/\C}\uuu,ud-df)$$
  is the classical HKR isomorphism extended by $\C\uuu$-linearity, i.e.,
\begin{equation}\label{tr215} 
\Big[ a_0 [a_1 | \cdots | a_j]\Big]  \mapsto\Big[\str(\frac{a_0da_1\cdots da_j}{j!})\Big].
\end{equation}
Since $\Psi(a_0[a_1|\cdots|a_j])=(-1)^{j+{j\choose 2}}a_0[a_j|\cdots|a_1]$, we see the commutativity of \eqref{xy1}.
\end{proof}

\subsection{The second rectangular diagram}\label{34second}

A shuffle product on  normalized Hochschild complexes $\overline{\Hoch}(Q)$ and $\overline{\Hoch}(R)$ of dg algebras $(Q,d_Q)$ and $(R,d_R)$
$$
\sh: \overline{\Hoch}(Q)\otimes_\C \overline{\Hoch}(R)\to \overline{\Hoch}(Q\otimes_\C R)
$$
is defined as follows. For two elements $a'_0[a'_1|a'_2|\ldots
|a'_n]\in\overline{\Hoch}(Q)$ and  $a''_0[a''_1|a''_2|\ldots |a''_m]\in \overline{\Hoch}(R)$ the
shuffle product is given by the formula:
\begin{multline}\label{shuff}
\sh(a'_0[a'_1|a'_2|\ldots |a'_n]\otimes
a''_0[a''_1|a''_2|\ldots
  |a''_m])\\
:=(-1)^{\ast}\cdot a'_0\otimes a''_0\,\sh_{n,m}[a'_1\otimes 1|\ldots|a'_n\otimes 1| 1\otimes a''_1|\ldots|1\otimes a''_m]
\end{multline}
where $\ast=|a''_0|(|sa'_1|+\ldots+|sa'_n|)$ recalling $|sa'_i|=|a'_i|-1$ and
\begin{multline*}
  \sh_{n,m}[x_1|\ldots|x_n|x_{n+1}|\ldots|x_{n+m}]\\
  :=\sum_{\sigma\in S(n,m)}\pm[x_{\sigma^{-1}(1)}|\ldots|
x_{\sigma^{-1}(n)}|
x_{\sigma^{-1}(n+1)}|\ldots|x_{\sigma^{-1}(n+m)}]
\end{multline*}
where $S(n,m)\subset S_{n+m}$ consists of the elements $\sigma$ such that $\sigma(i)<\sigma(j)$ whenever $1\leq i<j\leq n$ or $n+1\leq i<j\leq n+m$
. The sign in front of each
summand is computed according to the contribution of $(-1)^{(|sx|)(|sy|)}$ for the transposition
$[\,\ldots|x|y|\ldots\,]\to[\,\ldots|y|x|\ldots\,]$. The shuffle product $\sh$ commutes with the differential $b$; see \cite[Lemma 2.6]{BW3}. For normalized mixed complexes $(\overline{\Hoch}(Q)\uuu, b+uB)$ and $(\overline{\Hoch}(R)\uuu, b+uB)$ of dg algebras $(Q,d_Q)$ and $(R,d_R)$, a cyclic shuffle product is defined by
\begin{multline*}
\mathsf{Sh}(a'_0[a'_1|\ldots |a'_n]\otimes
a''_0[a''_1|\ldots|a''_m])\\
:=(-1)^{\ast\ast}(1\otimes 1)\csh_{n+1,m+1}[a'_0\otimes1|\ldots|a'_n\otimes1|1\otimes a''_0|\ldots|1\otimes a''_m]
\end{multline*}
where the operator $\csh$ is defined by the same formula as $\sh$ but with the sum extended over all the {\it cyclic} shuffles, i.e.,  one cyclically permutes $a'_0\otimes 1, \ldots,a'_n\otimes 1$ and $1\otimes a''_0, \ldots,1\otimes a''_m$ and shuffles the $a'\otimes 1$-terms with the $1\otimes a''$-terms, and $\ast\ast=|a'_0|+|sa'_1|+\ldots+|sa'_n|$ ; see \cite[Sect. 4.3.2]{Loday}. The operator $\sh+u\SH$ commutes with $b+uB$ and it gives the K\"unneth map which we denote by $-\tilde{\star}-$; see \cite{shk3, Loday}.

\begin{lemma}\label{lemma22}
  Assume that the only critical point of $f\in Q$ and $g\in Q$ is $\fm$. Then the following diagram
\begin{equation}\label{le220}
  \scalebox{0.8}{ \xymatrix{
  \OHN(\MF^\fm(Q,f)) \otimes_\C  \OHN(\MF^\fm(Q,-f)) \ar[rr]^-{\tr_{\nabla}^\fm \otimes \tr_{\nabla}^\fm} \ar[d]_{-\star-}
  && \R\Gamma_\fm (\Omega_{Q/\C}^\bu\uuu, ud-df) \otimes_\C \R\Gamma_\fm (\Omega_{Q/\C}^\bu\uuu, ud+df)  \ar[d]^{-\wedge-} \\
  \OHN(\MF^{\fm}(Q, 0))  \ar[rr]^{\tr_{\nabla}^{\fm}} && \R\Gamma_{\fm} (\Omega^{\bullet}_{Q/\C}\uuu,ud) }}
  \end{equation}
 commutes at the level of homology.
 \end{lemma}

\begin{proof}
 By taking the de Rham differential $d$ as $\nabla$, from  \cite[Theorem 4]{shk3} one has a commutative diagram
\begin{equation}\label{main223}
\scalebox{0.9}{ \xymatrix{
  \OHN(\cE_f) \otimes_\C  \OHN(\cE_g) \ar[rr]^-{\tr_{\nabla} \otimes \tr_{\nabla}} \ar[d]_{-\tilde{\star}-}
  && (\Omega_{Q/\C}^\bu\uuu, ud-df) \otimes_\C (\Omega_{Q/\C}^\bu\uuu, ud-dg)  \ar[d]^{-\tilde{\wedge}-} \\
  \OHN(\cE_f \otimes_\C  \cE_g)  \ar[rr]^(0.35){\tr_{\nabla}} &&  (\Omega^{\bullet}_{(Q\otimes_\C Q)/\C}\uuu,ud-d(f\otimes1)-d(1\otimes g)) }}
  \end{equation}
at the level of homology. Under the identifications in \eqref{support} and \eqref{s213}, \eqref{main223} implies
\begin{equation}
\scalebox{0.8}{ \xymatrix{
  \OHN(\MF^\fm(Q,f)) \otimes_\C  \OHN(\MF^\fm(Q,g)) \ar[rr]^-{\tr_{\nabla}^\fm \otimes \tr_{\nabla}^\fm} \ar[d]_{-\tilde{\star}-}
  && \R\Gamma_\fm (\Omega_{Q/\C}^\bu\uuu, ud-df) \otimes_\C \R\Gamma_\fm (\Omega_{Q/\C}^\bu\uuu, ud-dg)  \ar[d]^{-\tilde{\wedge}-} \\
  \OHN(\MF^{\fm\times\fm}(Q\otimes_\C Q, f\otimes 1+1\otimes g))  \ar[rr]^{\tr_{\nabla}^{\fm\times\fm}} && \R\Gamma_{\fm\times\fm} (\Omega^{\bullet}_{(Q\otimes_\C Q)/\C}\uuu,ud-d(f\otimes1)-d(1\otimes g)) }}
  \end{equation}
commutes at the level of homology. Letting $g=-f$, by composing the K\"unneth map with the map induced by the multiplication map $Q \otimes_\C Q \to Q$ and then denoting it by $-\star-$, one has \eqref{le220}.

\end{proof}

\subsection{The third  triangular diagram}\label{thirdsec}
One has a dg functor  $\MF^{\fm}(Q,0) \to \Perf_{\Z/2}(\C)$ by restriction of scalars along the structural map $ \C \to Q$ that induces a map 
$$
w: \OHN_*(\MF^\fm(Q,0)) \to \OHN_*(\Perf_{\Z/2}(\C)),
$$
where $\Perf_{\Z/2}(\C)$ denotes the dg category of
$\Z/2$-graded complexes of (not necessarily finitely dimensional) 
$\C$-vector spaces having finite dimensional homology. From a dg functor $\C\to\Perf_{\Z/2}(\C)$ where $\C$ is a dg category with one object, there is a canonical $\C\uuu$-linear isomorphism
$$
v : \C\uuu \xra{\cong} \OHN_0(\Perf_{\Z/2}(\C)) 
$$
given by $a \mapsto a[\,]$ where $\OHN_0(\C)=\C\uuu$ and $a_0[\,]$ means an element of a Hochschild complex of the form $a_0[a_1 | \cdots | a_n]$ with $n = 0$. We define an even degree map $\trace: \OHN_*(\MF^\fm(Q,0)) \to \C\uuu$ by
$$
\trace:= v^{-1} w.
$$
Recalling the notations in Section~\ref{endo25}, since $\id_K$ is $(b+uB)$-closed in the normalized Hochschild complex $\overline{\Hoch}(\Lambda)$, from the similar argument to \cite[Proposition 4.9]{BW3} we see that
$$
\trace(\id_K[\,]) = \dim_\C \rH_0(K) - \dim_\C \rH_1(K)=1.
$$

\begin{theorem}\cite[Proposition 4.13]{BW3} \label{thm23}
Let $\eta: \Lambda \to \C$ be the augmentation map that sends $e_i^*$ to $0$. The composition
\begin{equation}
\label{lamtrace}
\OHN_*(\Lambda)  \xra{(\ref{213})} \OHN_*(\MF^\fm(Q_\fm,0)) \xra{\trace} \C\uuu
\end{equation}
coincides with
\begin{equation}
\label{augmentation}
\OHN_*(\Lambda) \xra{\OHN_*(\eta)} \OHN_*(\C) \xra{\cong} \C\uuu,
\end{equation}
where the second map in (\ref{augmentation}) is the canonical isomorphism.
In particular, if $a_0[a_1 | \dots | a_n]$ is a cycle in $\OHN(\Lambda)$, where $n > 0$, the map (\ref{lamtrace}) sends $a_0[a_1 | \dots | a_n]$ to 0.
\end{theorem}

\begin{proof}
  The arguments for the case of the Hochschild complex in the proof of \cite[Proposition 4.13]{BW3} work for the case of the negative cyclic complex case word for word.
\end{proof}

\begin{theorem} \label{thm112}
Let $Q=\C[x_1,\dots,x_n]$ and $\fm=(x_1,\dots,x_n)$. Then the diagram 
$$
\xymatrix{  
\OHN_0(\MF^{\fm}(Q_\fm,0))  \ar[dr]_{(-1)^{\frac{n(n+1)}{2}} \mathsf{trace}} \ar[rr]^{\tr_\nabla^\fm} && 
\mathbb{H}_{2n} \R \Gamma_\fm(\Omega^\bu_{Q_\fm/\C}\uuu,ud)   \ar[dl]^{\on{\res}} \\
  & \C\uuu \\
}
$$
commutes where $n = \dim(Q_\fm)$.
\end{theorem}

     \begin{lemma}\cite[Lemma 4.38]{BW3}\label{lem5242}
       Let  $Q'=\C[x_1,\dots,x_n]$, $\fm=(x_1,\dots,x_n)$, $Q''=\C[x_{n+1},\dots,x_{n+m}]$, and $\fm''=(x_{n+1},\dots,x_{n+m})$.
       Let $Q = Q' \otimes_\C Q''$ and $\fm = \fm' \otimes_\C Q'' + Q' \otimes_\C \fm''$.
       If Theorem~\ref{thm112} holds for each of $(Q',\fm')$ and $(Q'',\fm'')$, then it also holds for $(Q, \fm$).
         \end{lemma}

  \begin{proof}  
The arguments for the case of the Hochschild complex in the proof of \cite[Lemma 4.38]{BW3} with replacement of \cite[Proposition 3.19]{BW3} with Lemma~\ref{lemma22} work for the case of the negative cyclic complex case word for word.
\end{proof}

   \begin{lemma}\label{mainlemma}
     Let $Q = \C[x]$, $\fm = (x)$, and $\cE := \End_{\MF^{(x)}(\C[x]_{(x)}, 0)}(K)$ where $\cE$ is the differential $\Z/2$-graded $Q_\fm$-algebra generated by odd degree elements $e, e^*$ and $(K,d_K)\in\MF^{\fm}(Q_m,0)$ is the Koszul complex on the $x$.  Let
   $$\varpi_j  =  e[\overbrace{e^* | e^* |   \cdots | e^*}^{j}] 
\in  \overline{\Hoch}(\cE).
$$
For each $\varpi_j$, we can construct $\phi_j:=\sum_{k=0}^{\infty}\varpi_{j,k}u^k$ such that
$$(b+uB)(\phi_j)=b\varpi_{j,0}\text{ where }\varpi_{j,0}:=\varpi_j\text{ and }\varpi_{j,k}\in  \overline{\Hoch}(\cE).$$
\end{lemma}

     \begin{proof}
       Consider the normalized Hochschild complex $\overline{\Hoch} (\cE) =  (\Hoch (\cE), b) / D$ where $D$ is a subcomplex of $\Hoch (\cE)$ generated by elements $a_0[a_1| ... | a_n]$ for which $a_i = c \cdot \mathrm{id}_K$ for $c\in Q_\fm$ and some $i\ge 1$.  Let $\omega_{j,0}:=\varpi_j$. In  the normalized Hochschild complex $\overline{\Hoch}(\cE)$, we claim that given $\omega_{j,0}$, one can construct $\omega_{j,k}\in \overline{\Hoch}(\cE)$ such that
\begin{equation}\label{lemc1eq}
  \begin{aligned}
    B(\omega_{j,k})&=b(\omega_{j,k+1})\text{ for }k,j\geq 0\\
    \omega_{j,k+1}&:=\sh(e^\ast[e|e]\otimes \varphi_{j,k+1})
\end{aligned}
  \end{equation}
where $\varphi_{j,k+1}$ is $B$-exact and $b$-closed. We prove this by induction on $k$. When $k=0$, let
$$\varphi_{j,1}:=-\frac{1}{j}B( e^\ast[\overbrace{e^* | e^* |   \cdots | e^*}^{j-1}] )= -[\overbrace{e^* | e^* |   \cdots | e^*}^{j}]. $$
Note that $\varphi_{j,1}$ is $B$-exact and $b$-closed. We let
$$\omega_{j,1}:=\sh(e^\ast[e|e]\otimes \varphi_{j,1}).$$
Observing that $ [\overbrace{e^* | e^* |   \cdots | e^*}^{j}]$ is $b$-closed and $b(e^\ast[e|e])=-[e]$, we see that
$$\begin{aligned}
  B(\omega_{j,0})&=\sh\big(b(e^\ast[e|e])\otimes -[\overbrace{e^* | e^* |   \cdots | e^*}^{j}] \big)\\
  &=b\circ\sh(e^\ast[e|e]\otimes -[\overbrace{e^* | e^* |   \cdots | e^*}^{j}] )\\
  &=b\circ\sh (e^\ast[e|e]\otimes\varphi_{j,1})=b(\omega_{j,1}),
  \end{aligned}$$
which establishes the case when $k=0$. Recall the following proposition.
\begin{proposition}\cite[Corollary 4.3.5]{Loday}\label{loco}
  For $x,y\in \overline{\Hoch}(A)$ the normalized Hochschild complex of a dg algebra $A$,
$$B\circ\sh(x\otimes B y)=\sh(Bx\otimes By).$$
  \end{proposition}

Suppose that we construct
$$\omega_{j,i}:=\sh(e^\ast[e|e]\otimes\varphi_{j,i})\text{ for }i=1,\dots,k-1$$
where  $\varphi_{j,i}$ is $B$-exact and $b$-closed. Then
\begin{align*}
  B(\omega_{j,k-1})= & B\big(\sh(e^\ast[e|e]\otimes\varphi_{j,k-1})\big)&  \\
 = &       \sh\big(B(e^\ast[e|e])\otimes\varphi_{j,k-1}\big)     &  \text{ by Proposition~\ref{loco}}    \\
 = & \sh(b(\zeta)\otimes\varphi_{j,k-1}) & \text{by \eqref{alignc3}} \\
 = & b\circ \sh(\zeta\otimes\varphi_{j,k-1}) &  \text{$\varphi_{j,k-1}$ is $b$-closed}\\
   = &  b\circ \sh\Big(\sh\big( e^\ast[e|e]\otimes\frac{-1}{3}([e^\ast|e]+[e|e^\ast])\big)\otimes\varphi_{j,k-1}\Big)& \text{by \eqref{alignc4}}\\
   = &   b\circ \sh\Big(\sh\big( e^\ast[e|e]\otimes\frac{-1}{3}B(e^\ast[e])\big)\otimes\varphi_{j,k-1}\Big) &  \\
   = &   b\circ \sh\Big(e^\ast[e|e]\otimes\sh\big(\frac{-1}{3}B(e^\ast[e])\otimes\varphi_{j,k-1}\big)\Big) & \text{by associativity} \\
   = &  b\circ \sh\Big(e^\ast[e|e]\otimes B\circ \sh\big(\frac{-1}{3}e^\ast[e]\otimes\varphi_{j,k-1}\big)\Big)   & \text{ by Proposition~\ref{loco}.}  \\
 \end{align*}
Note that
\begin{equation}\label{alignc3}
\begin{aligned}
  B(e^\ast[e|e])&=b(\zeta)\text{ where }\\
  \zeta&:=-(e^\ast[e^\ast|e|e|e]+e^\ast[e|e^\ast|e|e]+e^\ast[e|e|e^\ast|e]+e^\ast[e|e|e|e^\ast])
\end{aligned}
\end{equation}
and
\begin{equation}\label{alignc4}
-3\zeta=\sh\big( e^\ast[e|e]\otimes([e^\ast|e]+[e|e^\ast])\big).
  \end{equation}
Let $\omega_{j,k}:=\sh(e^\ast[e|e]\otimes\varphi_{j,k})$ where
$$\varphi_{j,k}:= B\circ \sh(\frac{-1}{3}e^\ast[e]\otimes\varphi_{j,k-1}).$$
Clearly, $\varphi_{j,k}$ is $B$-exact. Since $b([e^\ast|e]+[e|e^\ast])=0$, we see that
$$\begin{aligned}
  b(\varphi_{j,k})&=b\circ B\circ \sh(\frac{-1}{3}e^\ast[e]\otimes\varphi_{j,k-1})\\
  &=b\Big(\sh\big(\frac{-1}{3}([e^\ast|e]+[e|e^\ast])\otimes\varphi_{j,k-1}\big)\Big)\\
     &=\sh\big( \frac{-1}{3}b([e^\ast|e]+[e|e^\ast])\otimes\varphi_{j,k-1}\big)\\
    &=0.
\end{aligned}$$
Thus we construct $\omega_{j,k}\in \overline{\Hoch}(\cE)$. Letting $(-1)^k\omega_{j,k}:=\varpi_{j,k}$, we prove that there exists $\phi_j:=\sum_{k=0}^{\infty}\varpi_{j,k}u^k$ such that
$$(b+uB)(\phi_j)=b\varpi_{j,0}\text{ where }\varpi_{j,0}:=e[\overbrace{e^* | e^* |   \cdots | e^*}^{j}] 
\in  \overline{\Hoch}(\cE).$$

\end{proof}

\begin{proof}[Proof of Theorem \ref{thm112}]
       From Lemma~\ref{lem5242}, it suffices to prove Theorem~\ref{thm112} for the case when $Q = \C[x]$ and $\fm = (x)$.
       As in the beginning of Section~\ref{thirdsec}, let $K$ be the Koszul complex on $x$ with $\mathbb{Z}/2$-grading, $\cE := \End_{\MF^{(x)}(\C[x]_{(x)}, 0)}(K_{(x)})$, and $\Lambda$ be the dg $\C$-subalgebra of $\cE$ generated by the $e_i^*$. By \eqref{213}, we have a $\C\uuu$-linear isomorphism
$$
\C[y]\uuu \xra{\cong} \overline{HN}_0(\MF^{(x)}(\C[x]_{(x)}, 0))
$$
where the isomorphism is realized by
$$
y \mapsto \id_K [e^*] \in \overline{HN}(\Lambda) \subseteq \overline{HN}(\MF^{(x)}(\C[x]_{(x)},0))
$$
and more generally,
$$
y^j \mapsto j! \id_K[\overbrace{e^*| \cdots | e^*}^j] \text{  for $j \geq 0$}.
$$
Recall that the map $\tr_\nabla^\fm$ is induced by the following
\begin{multline} 
  \C[y]\uuu \xra{\cong}
  \overline{HN}_0(\cE) \xla{\cong} 
  \mathbb{H}_2 \R\Gamma_{(x)} (\overline{\Hoch}(\cE)\uuu,b+uB) \\
  \xra{\R \Gamma_{(x)} \tr_\nabla} 
\mathbb{H}_2 \R\Gamma_{(x)}(\Omega^\bu_{\C[x]_{(x)}/\C}\uuu,ud)
\end{multline}
and note that
$$
\R\Gamma_{(x)} (\overline{\Hoch}(\cE)\uuu,b+uB) = (\overline{\Hoch}(\cE)\uuu\oplus\overline{\Hoch}(\cE)\uuu[1/x]\cdot\alpha,\partial)
$$
where the differential is given by $\partial := b + uB+\alpha$ and $\alpha$ denotes left multiplication by $\alpha$. For $j \geq 0$, by abuse of notation let
$$
y^{(j)} := \frac{1}{j!} y^j = \id_K [\overbrace{e^* | e^* |   \cdots | e^*}^{j}],
$$
which is $(b+uB)$-closed and let
$$
\omega_j :=  e[\overbrace{e^* | e^* |   \cdots | e^*}^{j}] 
\in  \overline{\Hoch}(\cE)\uuu[1/x].
$$
From
$$
b(\omega_j)  = x y^{(j)} - y^{(j-1)}\text{ where }y^{(-1)} : = 0,$$
one sees that
$$
b\left( \frac{1}{x} \omega_j + \frac{1}{x^2} \omega_{j-1} + \cdots + \frac{1}{x^{j+1}} \omega_0\right)  = y^{(j)}.
$$
By Lemma~\ref{mainlemma} we can construct $\phi_j:=\sum_{k=0}^{\infty}\varpi_{j,k}u^k$
such that
$$(b+uB)(\phi_j)=b\varpi_{j,0}\text{ where }\varpi_{j,0}:=\omega_j.$$
For each $j  \geq 0$, let
$$
\eta_j := y^{j} + j! \alpha \left(\frac{1}{x} \phi_j + \frac{1}{x^2} \phi_{j-1} + \cdots + \frac{1}{x^{j+1}} \phi_0\right),
$$
which is $(b+uB+\alpha)$-closed. From this, we see that for each $j \geq 0$, the class $\eta_j\in \R\Gamma_{(x)} (\overline{\Hoch}(\cE)\uuu,b+uB)$ corresponds to $y^{j} \in \overline{\Hoch}(\cE)\uuu$ under the canonical map $\R\Gamma_{(x)}(\overline{\Hoch}(\cE)\uuu,b+uB) \to (\overline{\Hoch(\cE)}\uuu,b+uB)$ letting $\alpha=0$. Hence, the inverse of
$$
\mathbb{H}_2 \R\Gamma_{(x)}(\overline{\Hoch(\cE)}\uuu,b+uB) \xra{\cong} \overline{HN}_0(\cE) = \C[y]\uuu
$$
maps $y^j$ to the class of $\eta_j$ for each $j  \geq 0$. From \eqref{eqn: tr formula} recall that $\tr_\nabla^\fm$ sends $\theta_0[\theta_1| \cdots |\theta_n] \in (\overline{\Hoch}(\cE)\uuu,b+uB)$ to
$$
\sum_{J=0}^{\infty}\sum_{J=j_0+ \cdots +j_n} (-1)^{j_0+ \cdots +j_n} \frac{1}{(n + J)!} \str(\theta_0 (u\nabla^2+d_K')^{j_0} \theta_1' \cdots \theta_n' (u\nabla^2+d_K')^{j_n} )
$$
where $d_K'=[\nabla,d_K]$ and $\theta_i'=[\nabla,\theta_i]$; see \cite{BW,CKK}.
We can take the Levi-Civita connection on $K$ such that $\nabla^2=0$, $e' = 0$, and $(e^*)'= 0$ in the basis
$\{1, e\}$ of $K$. Since $d_K' = - e^* dx$, it becomes
$$
\sum_{J=0}^{1}\sum_{J=j_0+ \cdots +j_n} (-1)^{J} \frac{1}{(n + J)!} \str(\theta_0 (d_K')^{j_0} \theta_1' \cdots \theta_n' (d_K')^{j_n} ).
$$
Noting that from the specific expression of $\varpi_{j,k}$ in \eqref{lemc1eq}, we see that 
$$
\begin{aligned}
  \tr_\nabla^\fm(\phi_j) & = 0 \text{ for $j \geq 1$;} \\
   \tr_\nabla^\fm(\phi_0) & = \str(e) + \str(e e^* dx); \\
   \tr_\nabla^\fm(y^{j}) & = 0 \text{ for $j \geq 1$;} \\
   \tr_\nabla^\fm(y^{0}) & = \str(\id_K) + \str(e^* dx). \\
  \end{aligned}
$$
Since $\str(e e^*) = -1$, $\str(e^*) = 0$, $\str(e) = 0$, and $\str(\id_K) = 0$, we have  
$ \tr_\nabla^\fm(\phi_0) = - dx$, $ \tr_\nabla^\fm(\phi_j) = 0$, and $ \tr_\nabla^\fm(y^j) = 0$ for all $j\geq0$. Thus, for all $j \geq 0$, 
$$
\tr_\nabla^\fm(\eta_j) = - j! (\frac{\alpha}{x^{j+1}} \otimes dx ).
$$
On the other hand, by Theorem~\ref{thm23} we know that
$$\begin{aligned}
\mathsf{trace}(y^0) &= 1,\\ 
\mathsf{trace}(y^j) &= 0\text{ for all $j \geq 1$}.
\end{aligned}$$
Thus, we prove that $\res \circ \tr_\nabla^\fm = - \mathsf{trace}$. 
\end{proof}

\begin{remark}
  From the commutativity of Diagram~\eqref{outline} and the results of Shklyarov in \cite{shkl} and Brown and Walker in \cite{BW3}, with multiplication by $u^n$ the composition of right vertical maps in Diagram~\eqref{outline} is Saito's higher residue  pairings.
   In \cite{HK}, we reprove it without assuming  the commutativity of Diagram~\eqref{outline}.
\end{remark}

\section{Globalization of the canonical  pairing}\label{normal}

In this section, we globalize the canonical  pairing on the twisted de Rham cohomology. We closely follow the arguments for the Hodge cohomology in \cite{Kim} and generalize them to the twisted de Rham cohomology with some modification.

\subsection{Basic setup for globalization}

Let $X$ be a nonsingular variety of dimension $n$ over $\mathbb{C}$ and  $h$ be a nonzero regular function on $X$.
An object of a dg enhancement  $\MF (X, h)$ of the derived category  of matrix factorizations for $(X, h)$ is a $\mathbb{Z}/2$-graded vector bundle $E$ on  $X$ equipped with an $\cO_X$-linear homomorphism of degree $1$ $\delta _E : E \to E$ such that $\delta _E ^2 = h \cdot \mathrm{id}_E$. Note that the structure sheaf $\cO_X$ is concentrated in degree $0$. Assume that the critical locus of $h$ is set-theoretically in $h^{-1}(0)$ and proper over $\C$. It implies that $\MF(X,h)$ is proper and smooth. From the global version of the K\"{u}nneth formula and the following HKR type map; see \cite{shkl,BW3,CKK} for details
$$HP_* (   \MF (X\times X, \widetilde{h}))\to \HH^{-*} \big(X\times X,(\Omega _{X\times X}^{\bullet}\uu, ud-d\widetilde{h})\big),$$
one has a global Chern character $\chhp (\cO ^{\widetilde{h}}_{\Delta _X}) \in  \HH ^{0} \big(X\times X,(\Omega ^{\bullet}_{X\times X}\uu , ud-d\widetilde{h})\big)$ where
$\Delta_X \subset X \times X$ denotes the diagonal, $\widetilde{h} := h\ot 1 - 1 \ot h $, and $(\Delta_X)_*\cO_X:=\cO ^{\widetilde{h}}_{\Delta _X}$ from a natural functor 
$$(\Delta_X)_*:\MF(X,0)\to \MF(X\times X,\widetilde{h}).$$

Considering a vector bundle $E$ as an object in the derived category of coherent sheaves on $X$, one sees that from the HKR type isomorphisms,
$$\begin{aligned}
  HP_* \big(   D^b (\mathrm{coh} (X))\big)&\cong\HH^{-*} \big(X,(\Omega _{X}^{\bullet}, d)\big)\uu\\
  HN_* \big(  D^b (\mathrm{coh} (X))\big)&\cong \HH^{-*} \big(X,(\Omega _{X}^{\bullet}, d)\big)\uuu
  \end{aligned}$$
where $\Omega^\bullet_X:=\oplus_{p\in\Z}\Omega^p_X[p]$. The natural embedding from $\HH^{-*} \big(X,(\Omega _{X}^{\bullet}, d)\big)\uuu$ to $\HH^{-*} \big(X,(\Omega _{X}^{\bullet}, d)\big)\uu$ and the naturality between the Chern character maps implies that $\chhn(E)=\chhp(E)$ and from \cite{CKK}, it is given by
$$\chhp(E)= \str  \exp  (-R) $$
where $\nabla _i$ is a connection of $E|_{U_i}$ for a finite affine open covering $\fU = \{ U_i \}_{i \in I}$ of $X$ and 
\begin{align*}
  R &:=    \prod_{i\in I}  u\nabla_i ^2   + \prod _{i < j, i, j \in I } (\nabla _i - \nabla _j )\\
  \exp (-R) &:= \mathrm{id_E}  - R + R^2/2 - \cdots \cdots + (-1)^{\dim X} R^{\dim X}/(\dim X)!.
\end{align*} 
Here the products in the exponential are  Alexander-\v{C}ech-Whitney cup products in the \v{C}ech complex; see \cite{CKK} for details.  We define a {\em twisted de Rham valued Todd class} associated to a vector bundle $E$ over $X$ by
\begin{multline}\label{Toddformula}
  \td(E):=\sdet(\frac{-R}{1-\exp(R)})\in \bigoplus_{p}\rH^{0} \big(X,\mathcal{H}^p(\Omega _{X}^{\bullet}, d)[p]\big)\uuu\\
  \subset \bigoplus_{p}\rH^{0} \big(X,\mathcal{H}^p(\Omega _{X}^{\bullet}, d)[p]\big)\uu
  \end{multline}
where $\sdet$ is the superdeterminant. In particular, we let $\td (X)$ be the twisted de Rham valued Todd class of the tangent bundle of $X$.

Letting the {\em wedge product} $\cdot \wedge\cdot $  of twisted de Rham cohomology classes be the composition of the K\"{u}nneth map and the localization, we define a  pairing $\langle-,-\rangle$ by the composition of maps
\begin{multline}\label{mul2}   \HH ^{*} (X, (\Omega ^{\bullet}_X\uu, ud-dh))  \ti \HH ^{*} (X, (\Omega ^{\bullet}_X\uu, ud+dh))
 \xrightarrow{- \wedge - }   \HH ^{*} _Z(X, (\Omega ^{\bullet}_X\uu, ud))\\  
 \xrightarrow{\wedge \td (X) }    \HH ^{*} _Z(X, (\Omega ^{\bullet}_X\uu, ud))
 \xrightarrow{projection} \HH ^{0} _Z(X, \mathcal{H}^n(\Omega ^{\bullet}_X,d)[n])\uu\\
  \xrightarrow{natural} \rH^{0}_{c}(X, \mathcal{H}^n(\Omega ^{\bullet}_X,d )[n])\uu
 \xrightarrow{(-1)^{{n+1}\choose{2}} \int _X}    \C\uu
\end{multline}
where $Z$ is the critical locus of $h$ and $\int_X$ is the integration on properly supported cohomology $\rH^{\ast}_c$, which is described as follows. Let $f : X^n \to Y^m$ be a morphism between nonsingular varieties with dimension $n$ and $m$ respectively. Let $p\ge 0$ with $p-n+m \ge 0$. There is a natural homomorphism 
$$\mathscr{T} _f:\mathcal{H}^p(\Omega^\bullet_{X},d) [q]  \xrightarrow{}  f^* \mathcal{H}^{p-n+m}(\Omega _Y^{\bullet},d) [q-n+m]  \ot f^! \cO _Y .$$ 
It is known that there is a nonsingular complete variety $\bar{X}$ containing $X$ as an open subvariety 
and a proper morphism $\bar{f}: \bar{X} \to Y$ extending $f$. We let 
   $$\begin{aligned}
     nat:  \rH^{0}_{c}(X, \mathcal{H}^p(\Omega ^{\bullet}_X,d )[q])  &\to  \rH ^{0}(\bar{X}, \mathcal{H}^p(\Omega ^{\bullet}_{\bar{X}},d )[q])\\
     \int _{f}   :  \rH^{0}_{c}(X, \mathcal{H}^p(\Omega ^{\bullet}_X,d )[q]) &\to  \rH^{0}_c (Y, \mathcal{H}^{p-n+m}(\Omega ^{\bullet}_Y,d )[q-n+m])\text{ given by }\\
     \gamma &\mapsto \bar{f}_* ( \mathscr{T}_{\bar{f}} (nat (\gamma ))) . \end{aligned}$$
By abuse of notation, we still denote it by $\int _{f} $ after $\C\uu$-linear extension. We note that  $\int _{f} $ is independent of the choices of $\bar{X}$, an open immersion $X\hookrightarrow \bar{X}$, and an extension $\bar{f}$.
When $Y=\Spec \C$, we write $\int _{X}$ for $\int _{f}$; see \cite[Section 3.5]{Kim} for details.

Moreover, $\int_X$ satisfies the following properties whose proofs can be found in  \cite[Section 3.6]{Kim}; we note that they are proved in the case of Hodge cohomology $\rH^{0}_{c}(X, \mathcal{H}^p(\Omega ^{\bullet}_X,0 )[q])$ but the proofs in  \cite[Section 3.6]{Kim} can be translated word for word into the corresponding ones in the case of twisted de Rham cohomology $\rH^{0}_{c}(X, \mathcal{H}^p(\Omega ^{\bullet}_X,d )[q])\uu$. Namely,
\begin{enumerate}
\item[(i)] (Base change) Consider a fiber square diagram of smooth varieties over $\C$ where   $Y'$ is a closed point of $Y$: 
\begin{equation*}\label{diag: fiber diag}
  \xymatrix{ X' \ar[r]^{v} \ar[d] & X \ar[d]^f \\
    Y' \ar[r]_w & Y .  }
\end{equation*} 
Assume that $f$ is a flat, proper, and l.c.i morphism. Then 
\begin{equation*}
  \int _{X '} v^* (\gamma ) = w^* (\int _f \gamma ).
\end{equation*}
Moreover, assume that $f:X\to Y$ is a flat morphism but possibly non-proper and $Y$ is a connected smooth complete curve. Letting $\dim X=n$, we have
\begin{equation*}
  \int _{X '} v^* (\gamma ) = \int _f \gamma\in\C\text{ where } \gamma\in \rH^{0}_{c}(X, \mathcal{H}^{n-1}(\Omega ^{\bullet}_X,d )[n-1]).
\end{equation*}

\item[(ii)] (Functoriality) Let $X \xrightarrow{f} Y \xrightarrow{g} Z$ be morphisms of smooth varieties. 
Then \[
\int _{g \circ f} = \int _g \circ \int _f . \]

\item[(iii)] (Projection formula)
Let $f:X^n\to Y^m$ be a possibly non-proper morphism of smooth varieties with dimension $n$ and $m$ respectively.
Then \[ \int _f (f^* \sigma \wedge \gamma) = \sigma \wedge \int _f \gamma \]
for $\gamma \in \rH ^{0} _{cf}(X, \mathcal{H}^{n-m}(\Omega ^{\bullet}_X,d )[n-m])$ and $\sigma \in \rH ^{0} (Y, \mathcal{H}^{p}(\Omega ^{\bullet}_Y,d )[q])$.
\end{enumerate}

For a smooth and proper dg category $\cA$ over $\C$, there is another description of the canonical  pairing. That is, $\lan-,-\ran _{HP(\cA)\otimes HP(\cA^\op)}$ is a unique non-degenerate bilinear pairing $\lan-,-\ran$ satisfying 
\begin{equation}\label{eqn: char pairing}   
\sum _i \lan \gamma , T^i \ran \lan T_i, \gamma ' \ran = \lan \gamma, \gamma ' \ran  \text{ for every } \gamma \in HP(\cA )\text{ and }  \gamma '\in HP (\cA ^{op})
\end{equation}
where  $\Ch_{HP}(\Delta_\cA) = \sum_i T^i \ot T_i$ for some  $T^i \in HP (\cA ^{op} )$ and $\ T_i \in HP (\cA )$ via  the K\"unneth isomorphism; see \cite[page 1875]{PV: HRR}, \cite[Proposition 4.2]{Shk: HRR}, and \cite[Section 2.5]{Kim}. We will call Equation~\eqref{eqn: char pairing} the characteristic equation of the canonical  pairing.

In the setting of this section, we see that for $\gamma \in \HH ^{*} (X,(\Omega ^{\bullet}_{X}\uu, ud-dh))$ and $\gamma ' \in \HH ^{*} (X,(\Omega ^{\bullet}_{X}\uu, ud+dh)) $, we have 
\begin{align}
  \sum _i \lan \gamma , T^i \ran \lan T_i, \gamma ' \ran 
& =    \int _{X\ti X}  ( \gamma \ot \gamma ') \wedge \chhp (\cO ^{\widetilde{h}}_{\Delta _X}) \wedge (\td (X) \ot \td (X) )  \label{eqn: ch Delta}
\end{align}
where $\int _X \ot _\C \int _X = \int _{X \ti X} \circ \mathsf{kun} $ and
\[\chhp (\cO ^{\widetilde{h}}_{\Delta _X}) = \sum_i T^i \ot T_i  \in  \bigoplus _{q} \HH^{q} (X, (\Omega ^{\bullet}_X\uu, ud+dh)) \ot \HH ^{-q} (X, (\Omega ^{\bullet}_X\uu, ud-dh )). \]
To prove that the composition of maps in \eqref{mul2}, $\lan-,-\ran$, is the canonical pairing, by \eqref{eqn: char pairing} it suffices to show that the following is true:
\begin{equation}\label{bum2}
  \int _{X\ti X}  ( \gamma \ot \gamma ') \wedge \chhp(\cO ^{\widetilde{h}}_{\Delta _X}) \wedge (\td (X) \ot \td (X) )=\langle \gamma,\gamma'\rangle.
\end{equation}

\subsection{Deformation of normal cone}

The proof for Equation~\eqref{bum2} is similar to the arguments found in \cite{Kim}. For the sake of completeness, we present the arguments from \cite{Kim} with a small modification. First observe that there is a deformation space $M^{\circ}$ of $X\ti X$ to the normal cone of the diagonal $\Delta_{X}$ and a closed imbedding $\varphi:X\times\mathbb{P}^1\to M^\circ$ which is an extension of $\Delta\times \id_{\mathbb{A}^1}:X\times\mathbb{A}^1\to X\times X\times\mathbb{A}^1$; see \cite[Chapter 5]{fulton}:
$$\xymatrix{
X\times\mathbb{P}^1 \ar[dr]_{proj} \ar[r]^{\varphi} & M^\circ \ar[d]^{\varrho^\circ} \\
 & \mathbb{P}^1  }$$
The space $M^\circ$ is constructed as the complement of the blow-up of $X\times X\times\{\infty\}$ in the blow-up of $X\times X \times \PP ^1$ along the closed subscheme $\Delta _{X} \times \{\infty\}$. The preimages of a flat morphism $\varrho^\circ$ at the general points of $\PP^1$ are $X \ti X$ and the preimage of a special point $\{\infty\}$ of $\PP ^1$ is 
the normal cone $N_{\Delta _X / X\times X}$ which is isomorphic to the total space of the tangent bundle  $T_X$. We let
$M_{p}^{\circ}:=(\varrho^\circ)^{-1}(p)$ where $p$ is a closed point of $\mathbb{P}^1$ and isomorphisms
$g_p:X\times X\to M^{\circ}_{p}$ for $\infty\ne p\in \mathbb{P}^1$ and $g_\infty:T_X\to M^{\circ}_{\infty}$. It has a morphism $\psi: M^\circ\to X\times X$ factored as
$$\psi\circ g_\infty=\Delta\circ \pi$$
where $\pi: T_X\to X$ is the projection and $\Delta$ is the diagonal map.

\begin{lemma}\label{dagger1}\cite{Kim}
  For  $\rho :=  ( \gamma \ot \gamma ') \wedge (\td (X) \ot \td (X)) $ where  $\gamma \in \HH ^{*} (X,(\Omega ^{\bullet}_{X}\uu, ud-dh))$ and
  $\gamma ' \in \HH ^{*} (X,(\Omega ^{\bullet}_{X}\uu, ud+dh)) $, we have 
  \begin{equation}\label{eqdagger}
    \int _{X \ti X} \rho   \wedge \chhp (  (\varphi_{0})_* \cO _{X}) = \int _{T_{X} }   \pi ^* \Delta_X^* \rho  \wedge \chhp (  \mathrm{Kos} (s)).
    \end{equation}
  where $\mathrm{Kos}(s)$ is the Koszul complex $(\bigwedge ^{\bullet}  \pi^* T_X ^{\vee} , \iota _s )$ associated to the diagonal section $s$
  of  $\pi ^*T_X$ defined by $s(v) = (v, v) \in \pi^* T_X $ for $v\in T_X$ and $\Delta_X=\varphi_0:X\times \{0\}\to M^\circ$.
\end{lemma}

\begin{proof}
  The proof of Equation~\eqref{eqdagger} for $\gamma \in \HH ^{*} (X,(\Omega ^{\bullet}_{X}, -dh))$,
  $\gamma ' \in \HH ^{*} (X,(\Omega ^{\bullet}_{X}, dh)) $, and the case for $\chhh$ is in \cite{Kim}.

  The main observation in \cite[Section 3.3]{Kim} is that since $X\ti \PP^1$ and $M_{p}^{\circ}$ are Tor independent over $M^\circ$, one sees that
\begin{equation}
  \mathbb{L} g_p^* \varphi_*  \cO _{X\ti \PP ^1} \underset{qiso}{\sim} (\varphi_p)_* \cO _{X} ,
\end{equation}
i.e., they are quasi-isomorphic as coherent factorizations for $(M_{p}^{\circ}, -\psi^* \widetilde{h} |_{M_{p}^{\circ}})$.
In particular, since $\psi^*\widetilde{h}|_{M_{\infty}^{\circ} } = 0$ and $s$ is a regular section with the zero locus  $X \subset T_X$,
two factorizations $(\varphi_{\infty})_{ *}\cO _{X} $ and     $ \mathrm{Kos} (s)  $ are quasi-isomorphic to each other as coherent factorizations for $( T _{X} , 0)$. 

From the functoriality, the base change, the projection formula of the integration, and the main observation, we have

\begin{multline*}
  \int _{X \ti X} \rho   \wedge \chhp (  (\varphi_{0})_* \cO _{X})  =        \int _{X \ti X} \rho  \wedge \chhp (  \mathbb{L}g_0^* \varphi_* \cO _{X \ti \PP ^1} )          \\
=  \int _{X \ti X}  g_0^*(\psi^* \rho  \wedge  \chhp (  \varphi_* \cO _{X \ti \PP ^1} ))  \text{ by the functoriality of $\chhp$} \\
   =  \int _{T_{X}  }   g_{\infty}^*(\psi^* \rho  \wedge  \chhp (  \varphi_* \cO _{X \ti \PP ^1} )) =   \int _{T_{X} }   \pi ^* \Delta_X^\ast \rho  \wedge \chhp (  \mathrm{Kos} (s)).
 \end{multline*}
\end{proof}

From Lemma~\ref{dagger1}, we have
\begin{multline*}
  \int _{X\ti X}  ( \gamma \ot \gamma ') \wedge \chhp(\cO ^{\widetilde{h}}_{\Delta _X}) \wedge (\td (X) \ot \td (X) )\\
  = \int _{T_X}  \pi^* (\gamma \wedge \gamma ') \wedge \chhp (\mathrm{Kos} (s)) \wedge \pi^* \td (X) ^2 .
\end{multline*}

Let $\overline{\pi}$ denote the projection $\PP (T_X \oplus \cO_X)\to X$ and $\mathcal{Q}$ be the universal quotient bundle on $\PP (T_X \oplus \cO_X )$. We identify $\PP (T_X \oplus \cO_X)=\PP(T_X)\coprod T_X$ and $X=\PP(\cO_X)\hookrightarrow T_X$. The universal quotient bundle $\mathcal{Q}$ has a natural section $\overline{s}$ induced by a section $(0,-1)$ of $\overline{\pi}^* T_X \oplus \cO$. In particular, we see that $X=\{\overline{s}=0\}$. The Koszul complex $\mathrm{Kos} (\overline{s}):=(\bigwedge ^{\bullet} \mathcal{Q} ^{\vee} , \iota _{\overline{s}} )$ associated to $\overline{s}$ is a resolution of the sheaf $j_\ast\mathcal{O}_X$ where $j$ is the zero section imbedding of $X$ in $T_X$ followed by the canonical open imbedding of $T_X$ in  $\PP (T_X \oplus \cO_X)$: 
$$j:X\to T_X\subseteq \PP (T_X \oplus \cO_X).$$
One sees that the composition $\mathsf{quot} \circ (\mathrm{id}, 0) |_{T_X} : \pi ^* T_X \to \cQ |_{T_X}$ is an isomorphism sending $s$ to $\bar{s}|_{T_X}$.

\begin{lemma}\label{dagger2}
   Let $\gamma \in \HH ^{*} (X,(\Omega ^{\bullet}_{X}\uu, ud-dh))$ and $\gamma ' \in \HH ^{*} (X,(\Omega ^{\bullet}_{X}\uu, ud+dh)) $. Then
\begin{equation}\label{bum3}
  \int _{T_X}  \pi^* (\gamma \wedge \gamma ') \wedge \chhp (\mathrm{Kos} (s)) \wedge \pi^* \td (X) ^2 =(-1)^{\binom{n+1}{2}}  \int _{X} \gamma \wedge \gamma ' \wedge \td (X) .
\end{equation}
\end{lemma}

\begin{proof}
In \cite{Kim}, Equation~\eqref{bum3} is proved for $\gamma \in \HH ^{*} (X, (\Omega ^{\bullet}_{X}, -dh))$,
  $\gamma ' \in \HH ^{*} (X,(\Omega ^{\bullet}_{X}, dh)) $, $\chhh$, and $\tdhh$. We repeat the arguments in \cite{Kim} with a small modification.

By the base change, the functoriality, and the projection formula for the integration, we have
\begin{multline}\label{bum1}
  \text{ LHS of }\eqref{bum3}=\int _{\PP (T_X \oplus \cO_X )}  \overline{\pi}^*( \gamma \wedge \ \gamma ' \wedge \td (X)) \wedge \chhp (\mathrm{Kos} (\overline{s})) \wedge \td (\mathcal{Q})\\
= \int _{X}  \Big(\gamma \wedge \ \gamma ' \wedge \td (X)) \int_{\overline{\pi}}\chhp (\mathrm{Kos} (\overline{s})) \wedge \td (\mathcal{Q})\Big). 
\end{multline}
From Lemma~\ref{lem43}, one has $\int _{\overline{\pi}} \chhp (\mathrm{Kos} (\bar{s})) \td (\cQ) =(-1)^{ \binom{n+1}{2}}$.
Thus, we prove that
$$\int _{T_X}  \pi^* (\gamma \wedge \gamma ') \wedge \chhp (\mathrm{Kos} (s)) \wedge \pi^* \td (X) ^2 =(-1)^{\binom{n+1}{2}}  \int _{X} \gamma \wedge \gamma ' \wedge \td (X).$$
\end{proof}

\begin{lemma}\label{lem43}
  We have
  $$\int _{\overline{\pi}} \chhp (\mathrm{Kos} (\bar{s})) \td (\cQ) =    (-1)^{ \binom{n+1}{2}}.$$
\end{lemma}

\begin{proof}
  Consider a nonsingular complete variety $\bar{X}$ containing $X$. By abuse of notation,
  we let $\overline{\pi}$ be the projection $\PP (T_{\bar{X}} \oplus \cO_{\bar{X}})\to \bar{X}$, $\mathcal{Q}$ be the universal quotient bundle on  $\PP (T_{\bar{X}} \oplus \cO_{\bar{X}})$, and $\overline{s}$ be the natural section of $\mathcal{Q}$ defined as before. Note that $\bar{X}=\{\overline{s}=0\}$ and $\mathrm{Kos} (\overline{s})$ is the Koszul complex $(\bigwedge ^{\bullet} \mathcal{Q} ^{\vee} , \iota _{\overline{s}} )$. In this context, it suffices to show that
 $$\int _{\overline{\pi}} \chhp (\mathrm{Kos} (\bar{s})) \td (\cQ) =    (-1)^{ \binom{n+1}{2}}.$$

Since $\mathrm{Kos}(\overline{s})$ is a coherent factorization for $( \PP (T _{\bar{X}} \oplus \cO_{\bar{X}} ), 0)$, i.e., with the zero super-potential, from the naturality of Chern character maps and the canonical embedding
 $$\HH^\ast\big(\PP (T_{\bar{X}} \oplus \cO_{\bar{X}}),(\Omega_{ \PP (T _{\bar{X}} \oplus \cO_{\bar{X}} )}^{\bullet},d)\big)\uuu\hookrightarrow \HH^\ast\big(\PP (T_{\bar{X}} \oplus \cO_{\bar{X}}),(\Omega_{ \PP (T _{\bar{X}} \oplus \cO_{\bar{X}} )}^{\bullet},d)\big)\uu,$$
we have
$$\chhn(\mathrm{Kos}(\overline{s}))=\chhp(\mathrm{Kos}(\overline{s}))\in\HH^\ast\big(\PP (T_{\bar{X}} \oplus \cO_{\bar{X}}),(\Omega_{ \PP (T _{\bar{X}} \oplus \cO_{\bar{X}} )}^{\bullet},d)\big)\uuu.$$
Since the support of $\mathrm{Kos}(\overline{s})$ is $\bar{X}$, we have
\begin{equation*}
     \int _{\overline{\pi}} \chhn (\mathrm{Kos} (\overline{s})) \wedge \td (\mathcal{Q})=  \int _{\overline{\pi}}  \chhnx (\mathrm{Kos} (\overline{s})) \wedge \tdx (\mathcal{Q})
\end{equation*}
where $\chhnx$ is the localized Chern character; see \cite[Section 6.1 and 6.2]{CKK} for the formulae for $\chhn$ and $\chhnx$. Since $\mathrm{Kos} (\overline{s})$ is a resolution of the sheaf $j_\ast\mathcal{O}_{\bar{X}}$ whose support is $\bar{X}$, we see that $\cQ$ is trivial over the support.  One can take a local flat connection $\nabla_i$, i.e., $\nabla_i^2=0$, for each open set $U_i$ containing $\bar{X}$ to define
$\chhnx(\mathrm{Kos} (\overline{s}))$ and $\tdx(\cQ)$. That is,
$$\chhnx(\mathrm{Kos} (\overline{s})), \tdx(\cQ)\in\HH^{\ast}_{\bar{X}}\big(\PP (T_{\bar{X}} \oplus \cO_{\bar{X}}),(\Omega_{ \PP (T _{\bar{X}} \oplus \cO_{\bar{X}} )}^{\bullet},d)\big).$$
Furthermore, since $ \PP (T _{\bar{X}} \oplus \cO_{\bar{X}} )$ is projective, one can identify
$\chhnx(\mathrm{Kos} (\overline{s})), \tdx (\mathcal{Q})$ with $\chhhx(\mathrm{Kos} (\overline{s})), \tdhhx (\mathcal{Q})\in \HH^\ast_{\bar{X}}\big( \PP (T _{\bar{X}} \oplus \cO_{\bar{X}} ),(\Omega_{ \PP (T _{\bar{X}} \oplus \cO_{\bar{X}} )}^{\bullet},0)\big)$ respectively. Under this identification, we see that
\begin{equation*}
     \int _{\overline{\pi}} \chhnx (\mathrm{Kos} (\overline{s})) \wedge \tdx (\mathcal{Q})=  \int _{\overline{\pi}}  \chhhx (\mathrm{Kos} (\overline{s})) \wedge \tdhhx (\mathcal{Q})
\end{equation*}
We note that the integration of the right side is studied in \cite{Kim}. Since the support of $\mathrm{Kos}(\overline{s})$ is $\bar{X}$, we again have
\begin{equation*}
     \int _{\overline{\pi}} \chhhx (\mathrm{Kos} (\overline{s})) \wedge \tdhhx (\mathcal{Q})=  \int _{\overline{\pi}}  \chhh (\mathrm{Kos} (\overline{s})) \wedge \tdhh (\mathcal{Q})
\end{equation*}
Now the restriction of $\chhh (\mathrm{Kos} (\overline{s})) \wedge \tdhh (\mathcal{Q})$ to the fiber of $\overline{\pi}$ is nothing but $c_n(Q)$ where $Q$ is the tautological quotient bundle of $\mathbb{P}^{n}$ and $c_n(Q)$ is the $n$th Chern class; see \cite[Chapter 15]{fulton} for details, one calculates
$$\int _{\overline{\pi}} \chhh (\mathrm{Kos} (\bar{s})) \td (\cQ) = \int _{\PP ^n} c_n (Q) =    (-1)^{ \binom{n+1}{2}};\text{ see \cite[Section 3.6.6]{Kim}}.$$ 
\end{proof}

\begin{theorem}\label{tmo}
  Let $\gamma \in \HH ^{*} (X,(\Omega ^{\bullet}_{X}\uu, ud-dh))$ and $\gamma ' \in \HH ^{*} (X,(\Omega ^{\bullet}_{X}\uu, ud+dh)) $. Then
\begin{equation}
  \langle\gamma,\gamma'\rangle=(-1)^{\binom{n+1}{2}}  \int _{X} \gamma \wedge \gamma ' \wedge \td (X)
  \end{equation}
is a canonical  pairing.
\end{theorem}

\begin{proof}
  Combining Lemma~\ref{dagger1} and Lemma~\ref{dagger2}, we establish Equation~\eqref{bum2}. Thus, the pairing $\langle-,-\rangle$ in \eqref{mul2} is indeed a canonical  pairing.
\end{proof}

\begin{corollary}\label{tm61}
  For global matrix factorizations $P$ and  $Q$ in $\MF (X, h)$ over a smooth variety $X$, we have
 \begin{equation}
    \sum_{i \in \mathbb{Z}/2}  (-1)^i \dim  \RR ^i \Hom (P, Q)   
    =   (-1)^{ \binom{n+1}{2}}  \int_X \chhp (Q) \wedge \chhp (P)^\vee \wedge \td (X).
  \end{equation}
\end{corollary}

\begin{proof}
  From the abstract  Hirzebruch-Riemann-Roch formula, one knows that
$$\langle\,\Ch_{HP}(Q)\,,\,\Ch_{HP}(P)^\vee\,\rangle_{HP(\MF(X,h))\ot HP(\MF(X,-h))}=\Ch_{HP(\Perf (\mathbb{C}))}(\Hom(P,Q));$$
  see  \cite{Shk: HRR,TK} for details. Moreover, by Theorem~\ref{tmo}
  $$\begin{aligned}
    \langle\,\Ch_{HP}(Q)\,,\,\Ch_{HP}(P)^\vee&\,\rangle_{HP(\MF(X,h))\ot HP(\MF(X,-h))}=\langle\chhp(Q),\chhp(P)^\vee\rangle\\
    &=(-1)^{ \binom{n+1}{2}}  \int_X \chhp (Q)  \wedge \chhp (P)^\vee \wedge \td (X).
    \end{aligned}$$
  Using the Chern character formula in \cite{CKK} with taking a Levi-Civita connection, i.e., $\nabla^2=0$;
  see \cite[Remark 5.11]{BW}, which is always possible in $\Perf (\mathbb{C})$, one easily sees that
    $$\Ch_{HP(\Perf (\mathbb{C}))}(\Hom(P,Q))= \sum_{i \in \mathbb{Z}/2}  (-1)^i \dim  \RR ^i \Hom (P, Q).$$   
  \end{proof}

\begin{remark}
  By Corollary~\ref{tm61}.  we see that the twisted de Rham valued Chern character and Todd class, which are higher analogues of characteristic classes, do not produce a higher analogue of HRR theorem.  However, we remark that there are higher analogues of the ordinary Euler characteristic; see \cite{nir}. It will be an interesting problem to formulate a higher analogue of HRR theorem in this viewpoint.
 \end{remark}

\renewcommand\qedsymbol{\textbf{Q.E.D.}}

\end{document}